\documentclass{article}
\usepackage{amsthm}
\usepackage{amssymb}
\usepackage{amsmath}

\usepackage{diagrams}

\newtheorem{theorem}{Theorem}[section]

\newtheorem{definition}[theorem]{Definition}

\newtheorem{lemma}[theorem]{Lemma}

\newtheorem{proposition}[theorem]{Proposition}

\begin{document}

\title{Janelidze's Categorical Galois Theory as a step in the Joyal and Tierney result}
\author{Christopher Townsend}
\maketitle
\begin{abstract}
We show that a trivial case of Janelidze's categorical Galois theorem can be used as a key step in the proof of Joyal and Tierney's result on the representation of Grothendieck toposes as localic groupoids. We also show that this trivial case can be used to prove the general categorical Galois theorem by using a rather pleasing technical result about sliced adjunctions.
\end{abstract}

\section{Introduction}

The motivation for this paper came from trying to understand how Janelidze's categorical account of Galois Theory, which we here take to mean Theorem 5.1.24 of \cite{GaloisTheories}, relates to Joyal and Tierney's result on the representation of bounded toposes via localic groupoids (\cite{JoyT}). Although Chapter 7 of \cite{GaloisTheories} already provides such a relationship it is at the expense of moving to 2-categorical language and requires a re-statement of the categorical Galois Theorem itself. Further the author felt that his own Lemma 3.1, deployed in \cite{towgroth} as a step towards proving the Joyal and Tierney result, must be a known bit of abstract nonsense, given that it is a general result characterising the connected components adjunction of any groupoid in any cartesian category.

One of our conclusions is that one way round of Lemma 3.1 is an example of Janelidze's Galois Theorem, in fact a trivial one. Therefore the connection between categorical Galois Theory and Joyal and Tierney's result is easily made. What the author has noted along the way is that the categorical Galois theorem follows from the trivial case. The trivial case is also appealing because it provides not only criteria for when we have a connected components adjunction of a groupoid but is a characterization of the situation.

\subsection{Technical summary}

We now begin a more technical summary of our topic. Recall that Theorem 5.1.24 of \cite{GaloisTheories} asserts, in summary, that given an adjunction $\mathcal{D}\dashv \mathcal{C}: \mathcal{A} \pile{\rTo^{\mathcal{D}} \\ \lTo_{\mathcal{C}}} \mathcal{P}$ between two categories, both with pullbacks, together with some subcategories of the domain and codomain to which the adjunction restricts for every slice, then, given an effective descent morphism relative to one of these restricted slices of $\mathcal{A}$, there is a further subcategory of a slice of $\mathcal{A}$, referred to as those morphisms split by the effective descent morphism, which is equivalent to a full subcategory of the category of $\mathbb{G}$-objects for some groupoid $\mathbb{G}$ internal to $\mathcal{P}$. We label this the Relative Categorical Galois Theorem and note that the subcategories involved are determined by so called `admissible' classes of arrows.

If we forget about the subcategories/admissible arrows in the above formulation things get simpler. We label the resulting Categorical Galois Theorem as `Absolute'. It is the assertion that given an adjunction $\mathcal{D} \dashv \mathcal{C}:\mathcal{A} \pile{\rTo \\ \lTo} \mathcal{P}$ and $\sigma: S \rTo R$ an effective descent morphism of $\mathcal{A}$, then the full subcategory of $\mathcal{A}/R$ consisting of $\sigma$-split morphisms is equivalent to $[Gal[\sigma],\mathcal{P}]$ for some groupoid $Gal[\sigma]$ in $\mathcal{P}$ provided the counit of the adjunction, sliced at $S$, is an isomorphism and any object in the image of the right adjoint of the sliced adjunction, once precomposed with $\sigma$, is $\sigma$-split. Here a morphism with codomain $R$ is $\sigma$-split provided its pullback along $\sigma$ is fixed by the sliced adjunction. 

Let us now assume the context of the Absolute Categorical Galois Theorem but take $R=1$ (so that $\mathcal{A}$ and $\mathcal{P}$ must both be cartesian) and assume the sliced adjunction $\mathcal{A}/S \pile{\rTo \\ \lTo} \mathcal{P}/\mathcal{D}S$ is an equivalence. We will show below that it follows that every morphism to $1$ is $\sigma$-split and the counit of the adjunction sliced at $S$ is of course then always an isomorphism. Therefore by the Absolute Categorical Galois theorem, $\mathcal{A}$ is equivalent to $[Gal[S],\mathcal{P}]$ for some groupoid $Gal[S]$. We call this the Trivial Categorical Galois Theorem. It can in fact be proved directly:
\begin{proposition}\label{Trivial}
Given an adjunction $\mathcal{D} \dashv \mathcal{C}:\mathcal{A} \pile{\rTo \\ \lTo} \mathcal{P}$ between cartesian categories and an object $S$ of $\mathcal{A}$ such that $!:S \rTo 1$ is an effective descent morphism and $\mathcal{D}_S: \mathcal{A}/S \rTo \mathcal{P}/{\mathcal{D}S}$ is an equivalence. Then there exists a groupoid $\mathbb{G}$ internal to $\mathcal{P}$ such that $\mathcal{A}$ is equivalent to $[ \mathbb{G}, \mathcal{P}]$. 
\end{proposition}
\begin{proof}
Take $\mathbb{G}$ to be the image of the groupoid $(S \times S \pile{\rTo^{\pi_2} \\ \rTo_{\pi_1}}S, ...)$ under $\mathcal{D}$. This image determines a groupoid in $\mathcal{P}$ because the only pullbacks that need to be preserved are over $S$, which are preserved by assumption that the sliced adjunction is an equivalence. The category of algebras of the monad induced by the pullback adjunction $\Sigma_S \dashv S^* : \mathcal{A}/S \pile{\rTo \\ \lTo} \mathcal{A}$ is then equivalent, via $\mathcal{A}/S \simeq  \mathcal{P}/\mathcal{D}S$, to the category of $\mathbb{G}$-objects by construction. (Recall that a morphism is of effective descent if, by definition, its corresponding pullback functor is monadic.) 
\end{proof}

Our Trivial Categorical Galois Theorem is one round of Lemma 3.1 of \cite{towgroth} which also goes the other way round: if $\mathcal{D}\dashv \mathcal{C}$ is the connected components adjunction of a groupoid then there exists $\sigma: S \rTo 1$ of effective descent such that the adjunction, sliced at $S$, is an equivalence. This provides the connection of the title of the paper since Lemma 3.1 is used as a step in proving the Joyal and Tierney result. 

To provide a bit more colour on this connection we note that if $p: \mathcal{E} \rTo \mathcal{S}$ is a geometric morphism then it gives rise to an adjunction $\Sigma_p \dashv p^*: \mathbf{Loc}_{\mathcal{E}} \pile{ \rTo \\ \lTo} \mathbf{Loc}_{\mathcal{S}}$ between categories of locales over the domain and codomain toposes (the right adjoint is pullback in the category of toposes, treating a locale as a localic geometric morphism). If, further, $p$ is bounded with bound $B$ then the unique locale map $[\mathbb{N} \rTo B] \rTo 1$ from the locale of surjections from the naturals onto $B$, is an effective descent morphism and it can be shown, essentially from the definition of being a bound, that the sliced adjunction $\Sigma_{p_{[\mathbb{N} \rightarrow B]}} \dashv p^*_{[\mathbb{N} \rightarrow B]}$ is an equivalence. Hence the category of locales over $\mathcal{E}$ is equivalent to the category of $\mathbb{G}$-locales for some localic groupoid $\mathbb{G}$ over our base topos $\mathcal{S}$; the Joyal and Tierney result follows by restricting to discrete locales over $\mathcal{E}$. 

Aside from this connection to the Joyal and Tierney result, the other contribution of this paper is to show that the Trivial Categorical Galois Theorem can be used to prove the Absolute Categorical Galois Theorem.

\subsection{A comment on motivation}
We do not focus on examples of categorical Galois Theorem in action here and so do not demonstrate in this paper how our abstract categorical work relates to more familiar results about, say, fields, Galois groups or the fundamental group of a topological space. The book \cite{GaloisTheories}, which has proved invaluable to the author in his work on this topic, sets out in detail plenty of such examples, providing a proper motivation for the study of categorical Galois Theory.

\section{Slicing adjunctions}

In this section we recall some basic facts about sliced categories and adjunction between them. 

If $f: X \rTo Y$ is a morphism of a category $\mathcal{C}$ with pullbacks then there is an adjunction $\Sigma_f \dashv f^*: \mathcal{C}/X  \pile{ \rTo \\ \lTo} \mathcal{C}/Y$, where $f^*$ is pullback. This is the \emph{pullback adjunction} of $f$. We will write $Z_g$ for a typical object of $\mathcal{C}/X$; i.e. for $g:Z \rTo X$. So, for example, $\Sigma_f(Z_g) = Z_{fg}$. We write $Z_X$ for the projection $\pi_1 : X \times Z \rTo X$, an object of $\mathcal{C}/X$, and $\Sigma_X \dashv X^*$ for the pullback adjunction $\mathcal{C}/X \pile{ \rTo \\ \lTo} \mathcal{C}$ of $!: X \rTo 1$. An adjunction of this form is known as a \emph{slice}. All pullback adjunctions can be seen to be slices because for any $f: X \rTo Y$,  $\mathcal{C}/X$ is isomorphic to $(\mathcal{C}/Y)/X_f$. We observe that a morphism of $\mathcal{C}/X$ which is split in $\mathcal{C}$ is necessarily a regular monomorphism. To prove this observation say $n:Y_f \rTo Z_g$ is split in $\mathcal{C}$ by $k:Z \rTo Y$ (i.e. $kn=Id_Y$), then it is readily checked that $n$ is the equalizer of $Z_g \rTo^{(g,k)} Y_X \rTo^{Id_X \times n}Z_X$ and $Z_g \rTo^{(g,Id_Z)}Z_X$.

Any adjunction $L \dashv R: \mathcal{D} \pile{ \rTo^L \\ \lTo_R} \mathcal{C}$ can be sliced at any object $X$ of $\mathcal{C}$. The sliced adjunction $L_X \dashv R_X: \mathcal{D}/RX \pile{\rTo \\ \lTo } \mathcal{C}/X$ is given by $L_X(V_g) =$ `the adjoint transpose of $g$' and $R_X(Z_h)=RZ_{Rh}$. If we further assume that $\mathcal{D}$ has pullbacks then the adjunction can also be sliced at any object $W$ of $\mathcal{D}$. The sliced adjunction is written $L_W \dashv R_W: \mathcal{D}/W \pile{\rTo \\ \lTo } \mathcal{C}/LW$ and is given by $L_W(V_g) =LV_{Lg}$ and $R_W(Z_g)=$ `the pullback of $Rg$ along $\eta_W$', where $\eta$ is the unit of $L \dashv R$. If $\mathcal{C}$ also has pullbacks then these adjunctions relate nicely to one another: $L_W \dashv R_W$ is the composition of $\Sigma_{\eta_W} \dashv \eta_W^*$ followed by $L_{LW} \dashv R_{LW}$ and $L_X \dashv R_X$ is the composition of $L_{RX} \dashv R_{RX}$ followed by $\Sigma_{\epsilon_X} \dashv \epsilon_X^*$, where $\epsilon$ is the counit of $L \dashv R $.

Sliced adjunctions interact well with pullback adjunctions. The left adjoints of the following square of adjunctions clearly commute for any $g: W \rTo V$, a morphism of $\mathcal{D}$:
\begin{diagram}
\mathcal{D}/W & \pile{\rTo^{L_W} \\ \lTo_{R_W}} & \mathcal{C}/LW\\
\dTo^{\Sigma_g}  \uTo_{g^*} &        & \dTo^{\Sigma_{Lg}} \uTo_{(Lg)^*} \\
\mathcal{D}/V & \pile{\rTo^{L_V} \\ \lTo_{R_V}} & \mathcal{C}/LV \\
\end{diagram}
Therefore the right adjoints must also commute up to isomorphism.

If $L \dashv R: \mathcal{D} \pile{ \rTo \\ \lTo} \mathcal{C}$ is an adjunction between two categories such that $\mathcal{D}$ is cartesian and $\mathcal{C}$ has pullbacks, then $L_1 \dashv R_1: \mathcal{D}/1 \pile{ \rTo \\ \lTo} \mathcal{C}/L1$ is an adjunction between cartesian categories. Say for an object $W$ of $\mathcal{D}$, $!:W \rTo 1$ is an effective descent morphism and $L_W \dashv R_W$ is an equivalence of categories. Then by our introductory remarks on the direct proof of the Trivial Categorical Galois Theorem we know that there exists an internal groupoid $\mathbb{G}$ in $\mathcal{C}/L1$ such that $\mathcal{D}/1 \simeq [\mathbb{G}, \mathcal{C}/L1]$. Now $\mathbb{G}$ can be also be considered to be internal to $\mathcal{C}$ since pullbacks in $\mathcal{C}/L1$ are created in $\mathcal{C}$; we write $Gal[W]$ for $\mathbb{G}$ considered as internal to $\mathcal{C}$. As the groupoid $\mathbb{G}$ is constructed over $LW$, its $\mathbb{G}$-objects are over $LW$ from which it is clear that they are always over $L1$. From this it is clear that $[\mathbb{G},\mathcal{C}/L1]$ and $[Gal[W],\mathcal{C}]$ are the same category and so $\mathcal{D} \simeq [Gal[W], \mathcal{C}]$. This reasoning will be used below in our proof of the Absolute Categorical Galois Theorem.

We end this section with the following pleasing technical categorical lemma which will be a key step in our proof of the Absolute Categorical Galois Theorem to follow. The lemma shows that in certain circumstances checking that the unit of a sliced adjunction is an isomorphism when evaluated at a morphism, only requires knowledge of the domain of the morphism.    
\begin{lemma}\label{technical}
Let $L \dashv R: \mathcal{D} \pile{ \rTo \\ \lTo} \mathcal{C}$ be an adjunction between two categories with $\mathcal{D}$ cartesian. Let $W$ be an object of $\mathcal{D}$ with the property that the counit of $L \dashv R$, sliced at $W$, is an isomorphism (so that $L_W R_W \cong Id_{\mathcal{C}/LW}$). Assume for every object $X_{\phi}$ of $\mathcal{C}/LW$ that the unit of the adjunction sliced at $W$ is an isomorphism at $W^* \Sigma_W R_W (X_{\phi})$; i.e. $\eta^W_{W^* \Sigma_W R_W (X_{\phi})}$ is an isomorphism. Then for any object $V_f$ of $\mathcal{D}/W$, $\eta^W_{W^*V}$ is an isomorphism iff $\eta^W_{V_f}$ is an isomorphism.   
\end{lemma}
\begin{proof}
Say $\eta^W_{V_f}$ is an isomorphism. Then $V_f \cong R_W L_W V_f$. Therefore $W^*V \cong W^*\Sigma_W R_W X_{\phi}$ for some $X_{\phi}$ in $\mathcal{C}/LW$ and so $\eta^W_{W^*V}$ is an isomorphism. 

In the other direction say $\eta^W_{W^*V}$ is an isomorphism. Then $(\pi_1,\eta_{W \times V}):W \times V \rTo W \times_{RLW} RL(W \times V)$ has an inverse, which must be of the form $(\pi_1,l):W \times_{RLW} RL(W \times V) \rTo W \times V)$ for some $l:W \times_{RLW} RL(W \times V) \rTo V$. It is then easy to verify by naturality of $\eta$ that
\begin{eqnarray*}
W \times_{RLW} RLV \rTo^{Id_W \times RL(f,Id_V)} W \times_{RLW} RL(W \times V) \rTo^l V
\end{eqnarray*}
is a split (in $\mathcal{C}$) for $V \rTo^{(f, \eta_V)} W \times_{RLW} RLV$; i.e. we have that the unit $\eta^W_{V_f}$ is split in $\mathcal{D}$. Therefore, by the earlier observation, there is an equalizer diagram 
\begin{eqnarray*}
V_f \rTo^{\eta^W_{V_f}} R_W L_W V_f \pile{ \rTo^a \\ \rTo_b} W^*\Sigma_W R_W L_W V_f
\end{eqnarray*}
for two morphisms $a$ and $b$. But $L_W \eta^W_{V_f}$ must be an isomorphism because the counit, sliced at $W$, is assumed to be an isomorphism (use the triangular identities). Therefore $L_W a = L_W b$ and so they must have the same adjoint transpose; i.e. $\eta^W_{W^*\Sigma_W R_W L_W V_f}a = \eta^W_{W^*\Sigma_W R_W L_W V_f}b$. We also know that $\eta^W_{W^*\Sigma_W R_W L_W V_f}$ is an isomorphism by assumption and so $a=b$, which shows that $\eta^W_{V_f}$ is an isomorphism as required.     
\end{proof}

\section{Absolute Categorical Galois Theorem}

\begin{definition}
Given an adjunction $\mathcal{D} \dashv \mathcal{C}: \mathcal{A} \pile{\rTo \\ \lTo} \mathcal{P}$ between two categories, a morphism $\sigma: S \rTo R$ of $\mathcal{A}$ is said to be of \emph{Galois descent} provided (i) $\sigma^*: \mathcal{A}/R \rTo  \mathcal{A}/S$ is monadic, (ii) the counit of the adjunction $\mathcal{D} \dashv \mathcal{C}$, sliced at $S$, is an isomorphism; and, (iii) for every object $X_{\phi}$ of $\mathcal{P}/\mathcal{D}S$, the unit of the adjunction $\mathcal{D} \dashv \mathcal{C}$, sliced at $S$, is an isomorphism at $\sigma^*\Sigma_{\sigma}\mathcal{C}_SX_{\phi}$. 
\end{definition}

Given $\sigma: S \rTo R$, an object $A_a$ of $\mathcal{A}/R$ is said to be \emph{$\sigma$-split} provided the unit of $\mathcal{D} \dashv \mathcal{C}$, sliced at $S$, is an isomorphism when evaluated at $\sigma^*A_a$. Condition (iii) of the definition of Galois descent is that $\Sigma_{\sigma}\mathcal{C}_S X_{\phi}$ is $\sigma$-split for every object $X_{\phi}$ of $\mathcal{P}/\mathcal{D}S$. The full subcategory of $\mathcal{A}/R$ consisting of $\sigma$-split objects is written $Split_R(\sigma)$.

Notice for any morphism $\sigma:S \rTo R$ of Galois descent that by exploiting the fact that $\mathcal{C}_S\mathcal{D}_S\mathcal{C}_S \cong \mathcal{C}_S$ (since the counit is an isomorphism at the slice $S$) we have that an object $A_a$ is in $Split_R(\sigma)$ if and only if $\sigma^*A_a\cong \mathcal{C}_S X_{\phi}$ for some object $X_{\phi}$ of $\mathcal{P}/\mathcal{D}S$ (take $X_{\phi}=\mathcal{D}_S \sigma^* A_a$ one way round). From this it is clear that $Split_R(\sigma)$ has finite limits and so is cartesian (both $\mathcal{C}_S$ and $\sigma^*$ preserve finite limits as they are right adjoints). 

We now state the Absolute Categorical Galois Theorem:
\begin{theorem}\label{Absolute Categorical Galois Theorem}
Let $\mathcal{D} \dashv \mathcal{C}: \mathcal{A} \pile{\rTo \\ \lTo} \mathcal{P}$ be an adjunction between two categories, both with pullbacks. If we are given a Galois descent morphism $\sigma:S \rTo R$ in $\mathcal{A}$ then there exists a groupoid $Gal[\sigma]$ in $\mathcal{P}$ such that
\begin{eqnarray*}
Split_R(\sigma) \simeq [ Gal[\sigma],\mathcal{P}] \text{.}
\end{eqnarray*} 
\end{theorem}
\begin{proof}
We complete this proof by (a) showing that there is a right adjoint to $\mathcal{L}: Split_R(\sigma) \hookrightarrow \mathcal{A}/R \rTo^{\mathcal{D}_R} \mathcal{P}/\mathcal{D}R$, (b) showing that $S_{\sigma}$ is a member of $Split_R(\sigma)$, (c) verifying that ${\mathcal{L}}_{S_{\sigma}}:Split_R(\sigma)/S_{\sigma} \rTo \mathcal{P}/\mathcal{D}S$ is an equivalence; and, (d) showing that $!: S_{\sigma} \rTo 1$ is an effective descent morphism. Once (a), (b), (c) and (d) are established then since $\mathcal{P}/\mathcal{D}R$ is cartesian we can conclude by Proposition \ref{Trivial} that $Split_R(\sigma) \simeq [ \mathbb{G}, \mathcal{P}/\mathcal{D}R]$ which, by earlier comments, we know to be isomorphic to $[Gal[\sigma],\mathcal{P}]$ for some groupoid $Gal[\sigma]$ internal to $\mathcal{P}$.  

(a) The composition $\mathcal{P}/\mathcal{D}R \rTo^{\mathcal{C}_R}\mathcal{A}/R$ must factor through $Split_R(\sigma)$ because $\sigma^*\mathcal{C}_R \cong (\mathcal{D}\sigma)^* \mathcal{C}_S$ (recall our comment about sliced adjunctions interacting well with pullback adjunctions). Therefore we define $\mathcal{R}: \mathcal{P}/\mathcal{D}R \rTo Split_R(\sigma)$ to be the factorization of $\mathcal{C}_R$ through $Split_R(\sigma)$. Clearly $\mathcal{L} \dashv \mathcal{R}$.   

(b) Certainly $S_{\sigma}$ is a member of $Split_R(\sigma)$. This is because $S_{\sigma}=\Sigma_{\sigma}1=\Sigma_{\sigma}\mathcal{C}_S1$, as $\mathcal{C}_S$ preserves the terminal object, and we know that $\Sigma_{\sigma}\mathcal{C}_S 1$ is $\sigma$-split by definition of Galois descent. (Of course, $\mathcal{C}/X$ has a terminal object for any $X$, even if $\mathcal{C}$ doesn't; it is given by $X_{Id_X}$.)

(c)  First note that we can identify $Split_R(\sigma)/S_{\sigma}$ with the full category of $\mathcal{A}/S$ consisting of $B_b$ such that $\sigma b$ is $\sigma$-split (i.e. such that $\Sigma_{\sigma}(B_b) \in Ob(Split_R(\sigma)$). Now apply Lemma \ref{technical} to the adjunction $\mathcal{D}_R \dashv \mathcal{C}_R: \mathcal{A}/R \pile{\rTo \\ \lTo} \mathcal{P}/\mathcal{D}R$, with $W=S_{\sigma}$, to see that $Split_R(\sigma)/S_{\sigma}$ can be identified with all objects $B_b$ of $\mathcal{A}/S$ such that$\eta^S_{B_b}$ is an isomorphism. By combining this with the assumption that the counit, sliced at $S$, is an isomorphism it is clear that $Split_R(\sigma)/S_{\sigma} \simeq \mathcal{P}/\mathcal{L}(S_{\sigma})$.

(d) Proving that $(S_{\sigma})^*: Split_R(\sigma) \rTo Split_R(\sigma)/S_{\sigma}$ is monadic is relatively straightforward given the assumption that $\sigma$ is an effective descent morphism. Certainly $(R_{\sigma})^*$ reflects isomorphism as $\sigma^*$ does (the category $Split_R(\sigma)/S_{\sigma}$ is effectively a subcategory of $\mathcal{A}/S$). Given a parallel pair $f,g:A_a \pile{\rTo \\ \rTo} B_b$ that is $(S_{\sigma})^*$-split (in the sense of Beck's monadicity theorem) then it is certainly $\sigma^*$-split. As $\sigma^*$ is monadic we therefore know that there is a coequalizer of $f,g$, say $B_b \rTo^n Q_q$, such that $\sigma^*(Q_q)$ is isomorphic to the split coequalizer of $\sigma^*f, \sigma^*g$. But this split coequalizer diagram is in $Split_R(\sigma)/S_{\sigma}$ and from (c) we have clarified that the unit of $\mathcal{D}_S \dashv \mathcal{C}_S$ at any object in this category is an isomorphism, and so $Q_q$ is in $Split_R(\sigma)$.     
\end{proof} 

Proving in the other direction that the Trivial Categorical Galois Theorem follows from the Absolute Categorical Galois Theorem is simple enough. Say we are given $L \dashv R: \mathcal{D} \pile{ \rTo \\ \lTo} \mathcal{C}$ and an object $W$ of $\mathcal{D}$ with the property that $!: W \rTo 1$ is an effective descent morphism and $L_W: \mathcal{D}/W \rTo \mathcal{C}/LW$ is an equivalence. Then the unit and counit of $L_W \dashv R_W$ are isomorphisms at every object of $\mathcal{D}/W$ and $\mathcal{C}/LW$ respectively. Therefore every object of $\mathcal{D}$ is $W$-split, from which it is clear that $!:W \rTo 1$ is of Galois descent and $Split_1(W)$ is $\mathcal{D}$.

\section{Concluding comments}
The main conclusion of this paper is that we can prove the Joyal and Tierney result by applying a trivial case of the categorical Galois theory. Further we have shown that a technical categorical lemma allows us to prove non-relative categorical Galois theorem via the trivial case. In fact the trivial case characterises connected components adjunctions quite generally and indeed further characterisations exist which seem very suitable to understanding principal bundles in a general categorical context (see Proposition 7.4 of \cite{towHS}), giving further depth to our `trivial' case. 

This leaves open the question of whether relative categorical Galois theory can be similarly viewed. Indeed if we need to see the Joyal and Tierney result as being about toposes this suggests the need for a relative categorical Galois theorem. To clarify: given a bounded geometric morphism $p: \mathcal{E} \rTo \mathcal{S}$ the trivial categorical Galois theorem  shows us that the category of locales over $\mathcal{E}$ is equivalent to $[\mathbb{G},\mathbf{Loc}_{\mathcal{S}}]$ for some localic groupoid $\mathbb{G}$, but the actual Joyal and Tierney result requires us to restrict to discrete locales over $\mathcal{E}$. This is achieved by restricting the equivalence to local homeomorphisms and so on the surface this looks like this could follow by Theorem 5.1.24 of \cite{GaloisTheories} with admissible arrows taken to be local homeomorphisms. However a direct application is not possible as an open surjection (the relevant descent morphism) is not necessarily a local homeomorphism. A relative categorical Galois theorem can be formulated by extracting the categorical essence of the proof of Theorem 2.4, given right at the end of \cite{towgroth} and this then applied to obtain the Joyal and Tierney result as a fact about toposes. But this cannot be used to prove Theorem 5.1.24 of \cite{GaloisTheories} as the example of the abelianisation adjunction (Proposition 5.2.10 of \cite{GaloisTheories}) is not covered (essentially because the sliced counit is only an isomorphism once we restrict to surjections, which are the admissible arrows in this case).  

The topological case even suggests that a relative Galois categorical theory could give rise to two different groupoids based on the same functor. For any locally connected space $X$ there is an adjunction between the category of sheaves over $X$ and $\mathbf{Set}$, with the fundamental group coming from the Galois theory of this adjunction (provided certain separation axioms hold). However this is in contrast to the localic groupoid that arises from Joyal and Tierney which though trivial in this case (it is $Id,Id: X \pile{\rTo \\ \rTo }X $) nonetheless can be seen as coming from a Galois Theorem using the approach outlined in this paper. The right adjoint used to construct the fundamental group is a restriction to local homeomorphisms of the right adjoint used to construct the trivial localic groupoid $X \pile{ \rTo \\ \rTo} X$, so there is some sort of commonality and some relative-ness in this case, though I am not offering any non-trivial examples.  

Further purely categorical work could potentially make these relationships clearer.

\section{Acknowledgement}
The author would like to thank the anonymous referee of an earlier draft of this paper who disabused him of the idea that the abelianisation adjunction between groups and abelian groups is stably Frobenius. It is not!

\end{document}